\documentclass[reqno,11pt]{article}%
\usepackage{amsmath}
\usepackage{graphicx}
\usepackage{amsfonts}
\usepackage{amssymb}%
 \makeatletter
\newtheorem{theorem}{Theorem}

\newtheorem{conjecture}[theorem]{Conjecture}

\newtheorem{lemma}[theorem]{Lemma}

\newtheorem{proposition}[theorem]{Proposition}

\newenvironment{proof}[1][Proof]{\noindent{\textbf {#1}  }}  {\hfill$\Box$\bigskip}

\begin{document}

\title{The $p$-spectral radius of $k$-partite and $k$-chromatic uniform
hypergraphs\thanks{\textbf{AMS MSC:} 05C65; 05C35\textit{.}}
\thanks{\textbf{Keywords:} \textit{uniform hypergraph; }$p$-\textit{spectral
radius; }$k$\textit{-partite hypergraph; }$k$\textit{-chromatic hypergraph.}}}
\author{L. Kang\thanks{Department of Mathematics, Shanghai University, Shanghai,
200444, PR China. email: \textit{lykang@shu.edu.cn }} , V.
Nikiforov\thanks{Department of Mathematical Sciences, University of Memphis,
Memphis TN 38152, USA; email: \textit{vnikifrv@memphis.edu}} , and X.
Yuan\thanks{Corresponding author, Department of Mathematics, Shanghai
University, Shanghai, 200444, PR China. email:
\textit{xiyingyuan2007@hotmail.com}}}
\maketitle

\begin{abstract}
The $p$-spectral radius of a uniform hypergraph $G\ $of order $n$ is defined
for every real number $p\geq1$ as
\[
\lambda^{\left(  p\right)  }\left(  G\right)  =\max_{\left\vert x_{1}%
\right\vert ^{p}\text{ }+\text{ }\cdots\text{ }+\text{ }\left\vert
x_{n}\right\vert ^{p}\text{ }=\text{ }1}r!\sum_{\{i_{1},\ldots,i_{r}\}\in
E\left(  G\right)  }x_{i_{1}}\cdots x_{i_{r}}.
\]
It generalizes several hypergraph parameters, including the Lagrangian, the
spectral radius, and the number of edges. The paper presents solutions to
several extremal problems about the $p$-spectral radius of $k$-partite and
$k$-chromatic hypergraphs of order $n.$ Two of the main results are:

(I) Let $k\geq r\geq2,$ and let $G\ $be a $k$-partite $r$-graph of order $n.$
For every $p>1,$
\[
\lambda^{\left(  p\right)  }\left(  G\right)  <\lambda^{\left(  p\right)
}\left(  T_{k}^{r}\left(  n\right)  \right)  ,
\]
unless $G=T_{k}^{r}\left(  n\right)  ,$ where $T_{k}^{r}\left(  n\right)  $ is
the complete $k$-partite $r$-graph of order $n,$ with parts of size
$\left\lfloor n/k\right\rfloor $ or $\left\lceil n/k\right\rceil $.

(II) Let $k\geq2,$ and let $G\ $be a $k$-chromatic $3$-graph of order $n.$ For
every $p\geq1,$
\[
\lambda^{\left(  p\right)  }\left(  G\right)  <\lambda^{\left(  p\right)
}\left(  Q_{k}^{3}\left(  n\right)  \right)  ,
\]
unless $G=Q_{k}^{3}\left(  n\right)  ,$ where $Q_{k}^{3}\left(  n\right)  $ is
a complete $k$-chromatic $3$-graph of order $n,$ with classes of size
$\left\lfloor n/k\right\rfloor $ or $\left\lceil n/k\right\rceil $.

The latter statement generalizes a result of Mubayi and Talbot.

\end{abstract}

\section{\label{Def}Introduction}

In this paper we study the maximum $p$-spectral radius of $k$-partite and
$k$-chromatic uniform hypergraphs of given order.

Let us recall the definition of the $p$-spectral radius of graphs. Suppose
that $r\geq2$ and let $G$ be an $r$-uniform graph of order $n$. The
\emph{polynomial form} of $G$ is a multilinear function $P_{G}:\mathbb{R}%
^{n}\rightarrow\mathbb{R}^{1}$ defined for any vector $\left[  x_{i}\right]
\in\mathbb{R}^{n}$ as
\[
P_{G}\left(  \left[  x_{i}\right]  \right)  :=r!\sum_{\left\{  i_{1}%
,\ldots,i_{r}\right\}  \in E\left(  G\right)  }x_{i_{1}}\cdots x_{i_{r}}.
\]
Now, for any real number $p\geq1,$ the $p$\emph{-spectral radius} of $G$ is
defined as
\begin{equation}
\lambda^{\left(  p\right)  }\left(  G\right)  :=\max_{\left\vert
x_{1}\right\vert ^{p}+\cdots+\left\vert x_{n}\right\vert ^{p}=1}P_{G}\left(
\mathbf{x}\right)  . \label{defsa}%
\end{equation}
Note that $\lambda^{\left(  p\right)  }$ is a multifaceted parameter, as
$\lambda^{\left(  1\right)  }\left(  G\right)  $ is the Lagrangian of $G,$
$\lambda^{\left(  r\right)  }\left(  G\right)  $ is its spectral radius, and
$\lim_{p\rightarrow\infty}\lambda^{\left(  p\right)  }\left(  G\right)
n^{r/p}=r!e\left(  G\right)  $. The $p$-spectral radius has been introduced in
\cite{KLM13} and subsequently studied in \cite{NikA}, \cite{NikB}, and
\cite{NikC}.\medskip

Next, let us recall a few definitions about $k$-partite and $k$-chromatic
uniform hypergraphs. Let $k\geq r\geq2.$ An $r$-graph $G\ $is called
$k$\emph{-partite}\textbf{ }if its vertex set $V\left(  G\right)  $ can be
partitioned into $k$ sets so that each edge contains at most one vertex from
each set. An edge maximal $k$-partite $r$-graph is called \emph{complete }%
$k$\emph{-partite}. We write $T_{k}^{r}\left(  n\right)  $ for the complete
$k$-partite $r$-graph of order $n,$ with parts of size $\left\lfloor
n/k\right\rfloor $ or $\left\lceil n/k\right\rceil ;$ note that $T_{k}%
^{2}\left(  n\right)  $ is the Tur\'{a}n graph $T_{k}\left(  n\right)  $.

Further, an $r$-graph $G\ $is called $k$\emph{-chromatic}\textbf{ }if
$V\left(  G\right)  $ can be partitioned into $k$ sets so that no set contains
an edge. An edge maximal $k$-chromatic $r$-graph is called \emph{complete }%
$k$\emph{-chromatic}. We write $Q_{k}^{r}\left(  n\right)  $ for the complete
$k$-chromatic $r$-graph of order $n,$ with vertex sets of size $\left\lfloor
n/k\right\rfloor $ or $\left\lceil n/k\right\rceil ;$ note that $Q_{k}%
^{2}\left(  n\right)  =T_{k}\left(  n\right)  .$\medskip

For $2$-graphs, relations between the chromatic number and $\lambda^{\left(
p\right)  }$ have been long known. For example, if $G$ is a $k$-chromatic
$2$-graph of order $n,$ the result of Motzkin and Straus \cite{MoSt65} implies
that $\lambda^{\left(  1\right)  }\left(  G\right)  \leq1-1/k,$ Cvetkovi\'{c}
\cite{Cve72} has shown that $\lambda^{\left(  2\right)  }\left(  G\right)
\leq\left(  1-1/k\right)  n,$ and Edwards and Elphick \cite{EdEl83} improved
this to $\lambda^{\left(  2\right)  }\left(  G\right)  \leq\sqrt{2\left(
1-1/k\right)  e\left(  G\right)  }.$ Finally, Feng et al. \cite{FLZ07} have
shown that $\lambda^{\left(  2\right)  }\left(  G\right)  \leq\lambda^{\left(
2\right)  }\left(  T_{k}\left(  n\right)  \right)  .$ In fact, all these
inequalities have been improved by replacing the chromatic number with the
clique number of $G$.

However, for hypergraphs there are very few similar results. For example, if
$G$ is a $k$-chromatic $3$-graph of order $n$, Mubayi and Talbot \cite{MuTa08}
showed that
\begin{equation}
\lambda^{\left(  1\right)  }\left(  G\right)  \leq\lambda^{\left(  1\right)
}\left(  Q_{k}^{3}\left(  n\right)  \right)  . \label{MT}%
\end{equation}
Recently, in \cite{NikA} it was shown that if $G$ is a $k$-partite $r$-graph
of order $n$ and $p>1$, then
\begin{equation}
\lambda^{\left(  p\right)  }\left(  G\right)  \leq r!\binom{k}{r}%
^{1/p}k^{-r/p}e\left(  G\right)  ^{1-1/p}, \label{b1}%
\end{equation}
and
\begin{equation}
\lambda^{\left(  p\right)  }\left(  G\right)  \leq r!\binom{k}{r}%
k^{-r}n^{r-r/p}. \label{b2}%
\end{equation}
Also, if $G$ is a $k$-chromatic $r$-graph of order $n$ and $p>1,$ then%
\begin{equation}
\lambda^{\left(  p\right)  }\left(  G\right)  \leq\left(  1-\frac{1}{k^{r-1}%
}\right)  ^{1/p}\left(  r!e\left(  G\right)  \right)  ^{1-1/p}, \label{b3}%
\end{equation}
and
\begin{equation}
\lambda^{\left(  p\right)  }\left(  G\right)  \leq\left(  1-\frac{1}{k^{r-1}%
}\right)  n^{r-r/p}. \label{b4}%
\end{equation}
Bounds (\ref{b1})-(\ref{b4}) are quite tight, in view of the graphs $T_{k}%
^{r}\left(  n\right)  $ and $Q_{k}^{r}\left(  n\right)  $, but they can made
more precise, as shown in this paper.

The above list leaves quite a few gaps to be filled in. To our surprise, most
of these problems turned out to be astonishing challenges, much more
complicated than the corresponding results for $\lambda^{\left(  2\right)  }$
of $2$-graphs. We solved several problems to a satisfactory level, although
our proofs are generally quite long and technical. However, we could not solve
two central problems stated below as Conjectures \ref{co1} and \ref{co2}. It
is evident that new methods are necessary to attack these conjectures, and we
hope to have prepared some ground for them.\medskip

We proceed with statement and discussion of our main results, first for
$\lambda^{\left(  1\right)  }\left(  G\right)  $ of $k$-partite $r$-graphs.

\begin{theorem}
\label{MSt}Let $k\geq r\geq2,$ and let $G$ be a $k$-partite $r$-graph of order
$n$ with partition sets $V_{1},\ldots,V_{k}.$ Then
\begin{equation}
\lambda^{\left(  1\right)  }\left(  G\right)  \leq r!\binom{k}{r}k^{-r}.
\label{l1}%
\end{equation}

(I) If $\left[  x_{i}\right]  $ is a positive $n$-vector such that
$x_{1}+\cdots+x_{n}=1$ and
\begin{equation}
\lambda^{\left(  1\right)  }\left(  G\right)  =P_{G}\left(  \left[
x_{i}\right]  \right)  =r!\binom{k}{r}k^{-r}, \label{max1}%
\end{equation}
then $G$ is complete $k$-partite and%
\begin{equation}
\sum_{i\in V_{j}}x_{i}=\frac{1}{k},\text{ \ }j=1,\ldots,k. \label{c1}%
\end{equation}

(II) If $G$ is complete $k$-partite and $\left[  x_{i}\right]  $ is a
nonnegative $n$-vector satisfying (\ref{c1}), then (\ref{max1}) holds.
\end{theorem}

Clause \emph{(II) }of\emph{ }Theorem \ref{MSt} shows that there are many
non-isomorphic $r$-graphs achieving equality in (\ref{l1}). However, this is
not the case if $p>1,$ as shown in the following theorem.

\begin{theorem}
\label{th1}Let $k\geq r\geq2,$ and let $G\ $be a $k$-partite $r$-graph of
order $n.$ For every $p>1,$
\[
\lambda^{\left(  p\right)  }\left(  G\right)  <\lambda^{\left(  p\right)
}\left(  T_{k}^{r}\left(  n\right)  \right)  ,
\]
unless $G=T_{k}^{r}\left(  n\right)  .$
\end{theorem}

Although Theorem \ref{th1} is as good as one can get, it is also useful to
have explicit bounds which are close to the best possible one. Thus, for
reader's sake we shall give self-contained proofs of bounds (\ref{b1}) and
(\ref{b2}).

\begin{theorem}
\label{th2}Let $k\geq r\geq2,$ and let $G\ $be a $k$-partite $r$-graph of
order $n.$ If $p>1,$ then
\[
\lambda^{\left(  p\right)  }\left(  G\right)  \leq r!\binom{k}{r}%
^{1/p}k^{-r/p}e\left(  G\right)  ^{1-1/p}.
\]
Also, if $p>1,$ then
\begin{equation}
\lambda^{\left(  p\right)  }\left(  G\right)  <r!\binom{k}{r}k^{-r}n^{r-r/p},
\label{in1}%
\end{equation}
unless $k|n$ and $G=T_{k}^{r}\left(  n\right)  .$
\end{theorem}

Note that Theorem \ref{th2} requires that $p>1,$ as the conditions for
equality are different from those for $\lambda^{\left(  1\right)  }\left(
G\right)  ,$ as listed in Theorem \ref{MSt}.\medskip

We continue with problems for $k$-chromatic graphs, which are considerably
more difficult. The first result extends the bound of Mubayi and Talbot
(\ref{MT}).

\begin{theorem}
\label{th3}Let $k\geq2$, and let $G\ $be a $k$-chromatic $3$-graph of order
$n.$ For every $p\geq1,$
\[
\lambda^{\left(  p\right)  }\left(  G\right)  <\lambda^{\left(  p\right)
}\left(  Q_{k}^{3}\left(  n\right)  \right)  ,
\]
unless $G=Q_{k}^{3}\left(  n\right)  .$
\end{theorem}

Note again that Theorem \ref{th3} is very precise, but not explicit; however,
it is useful to have explicit bounds that are close to the best possible one,
like those given in the next theorem. Recall that $K_{n}^{r}$ stands for the
complete $r$-graph of order $n.$

\begin{theorem}
\label{th4}Let $k\geq2,$ let $G\ $be a $k$-chromatic $3$-graph of order $n,$
and let $p\geq1.$

(I) If $n\leq2k,$ then
\[
\lambda^{\left(  p\right)  }\left(  G\right)  <3!\binom{n}{3}n^{-3/p},
\]
unless $G=K_{n}^{3}.$

(II) If $n>2k,$ then
\[
\lambda^{\left(  p\right)  }\left(  G\right)  <3!\left(  \binom{n}{3}%
-k\binom{n/k}{3}\right)  n^{-3/p},
\]
unless $k|n$ and $G=Q_{k}^{3}\left(  n\right)  .$
\end{theorem}

It is immediate to extend clause (I) of Theorem \ref{th4} for $r>3$ and
$n\leq\left(  r-1\right)  k.$ Indeed, if $n\leq\left(  r-1\right)  k,$ then
\[
\lambda^{\left(  p\right)  }\left(  Q_{k}^{r}\left(  n\right)  \right)
=\lambda^{\left(  p\right)  }\left(  K_{n}^{r}\right)  =r!\binom{n}{r}%
n^{-r/p}.
\]
Hence, for every $r$-graph $G$ of order $n$, $\lambda^{\left(  p\right)
}\left(  G\right)  \leq\lambda^{\left(  p\right)  }\left(  Q_{k}^{r}\left(
n\right)  \right)  ,$ with equality holding if and only if $G=K_{n}^{r}.$ We
arrive thus at the following proposition.

\begin{proposition}
\label{pth4}Let $k\geq2$, and let $G\ $be a $k$-chromatic $r$-graph of order
$n\leq\left(  r-1\right)  k.$ For every $p\geq1,$
\[
\lambda^{\left(  p\right)  }\left(  G\right)  <r!\binom{n}{r}n^{-r/p},
\]
unless $G=K_{n}^{r}.$
\end{proposition}

The above observations show that for a meaningful generalization of Theorems
\ref{th3} and \ref{th4} we should require that $n>\left(  r-1\right)  k$.
Unfortunately, our methods are not good to tackle such generalization and so
we state two conjectures instead.

\begin{conjecture}
\label{co1}Let $k\geq2$, and let $G\ $be a $k$-chromatic $r$-graph of order
$n>\left(  r-1\right)  k.$ For every $p\geq1,$
\[
\lambda^{\left(  p\right)  }\left(  G\right)  <\lambda^{\left(  p\right)
}\left(  Q_{k}^{r}\left(  n\right)  \right)  ,
\]
unless $G=Q_{k}^{r}\left(  n\right)  .$
\end{conjecture}

\begin{conjecture}
\label{co2}Let $k\geq2,$ let $G\ $be a $k$-chromatic $r$-graph of order
$n>\left(  r-1\right)  k.$ For every $p\geq1,$%
\[
\lambda^{\left(  p\right)  }\left(  G\right)  <r!\left(  \binom{n}{r}%
-k\binom{n/k}{r}\right)  n^{-r/p},
\]
unless $k|n$ and $G=Q_{k}^{r}\left(  n\right)  .$
\end{conjecture}

Let us note that the difficulty of Conjectures \ref{co1} and \ref{co2} lies in
their level of precision. If cruder estimates are acceptable, then the simple
bounds (\ref{b3}) and (\ref{b4}) are good enough and are asymptotically tight.
For reader's sake we give self-contained proofs of these bounds.

\begin{theorem}
\label{th5}Let $k\geq r\geq2,$ and let $G\ $be a $k$-chromatic $r$-graph of
order $n.$ If $p\geq1,$ then
\begin{equation}
\lambda^{\left(  p\right)  }\left(  G\right)  \leq\left(  1-\frac{1}{k^{r-1}%
}\right)  ^{1/p}\left(  r!e\left(  G\right)  \right)  ^{1-1/p}, \label{in4}%
\end{equation}
and
\begin{equation}
\lambda^{\left(  p\right)  }\left(  G\right)  <\left(  1-\frac{1}{k^{r-1}%
}\right)  n^{r-r/p}. \label{in5}%
\end{equation}

\end{theorem}

In the remaining part of the paper we prove Theorems \ref{MSt}-\ref{th5}.

\section{Proofs}

In the course of our proofs we shall use a number of classical inequalities.
Among those are the Power Mean inequality (PM inequality), the Arithmetic Mean
- Geometric Mean inequality (AM-GM inequality), the Bernoulli and the
Maclaurin inequalities; for reference material, we refer the reader to
\cite{HLP88}.

For background on hypergraphs we refer the reader to \cite{Ber87}. As usual,
if $G$ is an $r$-graph of order $n$ and $V\left(  G\right)  $ is not defined
explicitly, it is assumed that $V\left(  G\right)  =[n]=\left\{
1,\ldots,n\right\}  ;$ this assumption is crucial for our notation.

All required facts about the $p$-spectral radius are given below. Additional
reference material can be found in \cite{NikA} and \cite{NikB}. In particular,
if $G$ is an $r$-graph of order $n$ and $\left[  x_{i}\right]  $ is an
$n$-vector such that $\left\vert x_{1}\right\vert ^{p}+\cdots+\left\vert
x_{n}\right\vert ^{p}=1$ and $\lambda^{\left(  p\right)  }\left(  G\right)
=P_{G}\left(  \left[  x_{i}\right]  \right)  ,$ then $\left[  x_{i}\right]  $
will be called an \emph{eigenvector }to $\lambda^{\left(  p\right)  }\left(
G\right)  .$ Clearly, $\lambda^{\left(  p\right)  }\left(  G\right)  $ always
has a nonnegative eigenvector.

The following lemma is useful for well-structured graphs, in particular for
complete partite and complete chromatic graphs. It can be traced back to
\cite{KLM13}.

\begin{lemma}
\label{eqth}Let $G$ be an $r$-graph of order $n$ with $E\left(  G\right)
\neq\varnothing,$ and let $u$ and $v$ be vertices of $G$ such that the
transposition of $u$ and $v$ is an automorphism of $G.$ If $p>1$ and $\left[
x_{i}\right]  $ is an eigenvector to $\lambda^{\left(  p\right)  }\left(
G\right)  ,$ then $x_{u}=x_{v.}$.
\end{lemma}

\begin{proof}
Note that
\[
P_{G}\left(  \left[  x_{i}\right]  \right)  =x_{u}A+x_{v}A+x_{u}x_{v}B+C,
\]
where $A,B,C$ are independent of $x_{u}$ and $x_{v}.$ Assume that $x_{u}\neq
x_{v}$ and define a vector $\left[  x_{i}^{\prime}\right]  $ such that
\[
x_{u}^{\prime}=x_{v}^{\prime}=\frac{x_{u}+x_{v}}{2},\text{ and }x_{i}^{\prime
}=x_{i}\text{ if \ }i\in\left[  n\right]  \backslash\left\{  u,v\right\}  .
\]
Since $p>1,$ the PM inequality implies that $\left\vert x_{1}^{\prime
}\right\vert ^{p}+\cdots+\left\vert x_{n}^{\prime}\right\vert ^{p}<\left\vert
x_{1}\right\vert ^{p}+\cdots+\left\vert x_{n}\right\vert ^{p}=1,$ while
\[
P_{G}\left(  \left[  x_{i}^{\prime}\right]  \right)  -P_{G}\left(  \left[
x_{i}\right]  \right)  =\frac{\left(  x_{u}-x_{v}\right)  ^{2}}{4}B\geq0,
\]
and so,
\[
\lambda^{\left(  p\right)  }\left(  G\right)  \geq\frac{P_{G}\left(  \left[
x_{i}^{\prime}\right]  \right)  }{\left\vert \left[  x_{i}^{\prime}\right]
\right\vert _{p}^{r}}>P_{G}\left(  \left[  x_{i}^{\prime}\right]  \right)
\geq P_{G}\left(  \left[  x_{i}\right]  \right)  =\lambda^{\left(  p\right)
}\left(  G\right)
\]
a contradiction, completing the proof of Lemma \ref{eqth}.
\end{proof}

\subsection{Proof of Theorem \ref{MSt}}

\begin{proof}
Let $\mathbf{x}$ be a nonnegative $n$-vector such that $\left\vert
\mathbf{x}\right\vert _{1}=1,$ $\lambda^{\left(  1\right)  }\left(  G\right)
=P_{G}\left(  \mathbf{x}\right)  ,$ and $\mathbf{x}$ has minimum number of
positive entries. Let $m$ be the number of positive entries of $\mathbf{x}.$
We shall show that $m\leq k$. Indeed, if $m>k$, then $\mathbf{x}$ has two
positive entries $x_{i}$ and $x_{j}$ belonging to the same partition set.
Since no edge contains both vertices $i$ and $j,$ we see that
\[
P_{G}\left(  \mathbf{x}\right)  =x_{i}\frac{\partial P_{G}\left(
\mathbf{x}\right)  }{\partial x_{i}}+x_{j}\frac{\partial P_{G}\left(
\mathbf{x}\right)  }{\partial x_{j}}+S,
\]
where $S$ does not depend on $x_{i}$ or $x_{j}.$ By symmetry, we assume that
$\frac{\partial P_{G}\left(  \mathbf{x}\right)  }{\partial x_{i}}\geq
\frac{\partial P_{G}\left(  \mathbf{x}\right)  }{\partial x_{j}}$ and define
the vector $\mathbf{x}^{\prime}$ by
\[
x_{i}^{\prime}=x_{i}+x_{j},\text{ }x_{j}^{\prime}=0,\text{ and }x_{s}^{\prime
}=x_{s}\text{ for }s\in\left[  n\right]  \backslash\left\{  i,j\right\}  .
\]
We see that $\left\vert \mathbf{x}^{\prime}\right\vert _{1}=1$ and
\[
P_{G}\left(  \mathbf{x}^{\prime}\right)  -P_{G}\left(  \mathbf{x}\right)
=x_{j}\left(  \frac{\partial P_{G}\left(  \mathbf{x}\right)  }{\partial x_{i}%
}-\frac{\partial P_{G}\left(  \mathbf{x}\right)  }{\partial x_{j}}\right)
\geq0.
\]
It follows that $P_{G}\left(  \mathbf{x}^{\prime}\right)  =P_{G}\left(
\mathbf{x}\right)  ,$ but $\mathbf{x}^{\prime}$ has only $m-1$ positive
entries, contradicting the choice of $\mathbf{x}.$ Hence $m\leq k$. By
symmetry, let $x_{1},\ldots,x_{m}$ be the positive entries of $\mathbf{x}$.
Now, using Maclaurin's inequality, we see that
\[
P_{G}\left(  \mathbf{x}\right)  \leq r!\sum_{1\leq i_{1}<\cdots<i_{r}\leq
m}x_{i_{1}}\cdots x_{i_{r}}\leq r!\binom{m}{r}\left(  \frac{1}{m}\sum
_{i=1}^{m}x_{i}\right)  ^{r}=r!\binom{m}{r}m^{-r}\leq r!\binom{k}{r}k^{-r}.
\]
This proves the bound (\ref{l1}).

Next, we prove (I). It is clear that $G$ is complete $k$-partite. Next
for\ $j=1,\ldots,k$, let
\[
y_{j}=\sum_{i\in V_{j}}x_{i},
\]
and using Maclaurin's inequality, we find that
\[
P_{G}\left(  \mathbf{x}\right)  =r!\sum_{1\leq i_{1}<\cdots<i_{r}\leq
k}y_{i_{1}}\cdots y_{i_{r}}\leq r!\binom{k}{r}\left(  \frac{1}{k}\sum
_{i=1}^{k}y_{i}\right)  ^{r}=r!\binom{k}{r}k^{-r}.
\]
The condition for equality of Maclaurin's inequality implies that $y_{j}=1/k$
for $j=1,\ldots,k,$ completing the proof of (I). To prove (II), it is enough
to notice that%
\[
\lambda^{\left(  1\right)  }\left(  G\right)  \geq P_{G}\left(  \mathbf{x}%
\right)  =r!\sum_{1\leq i_{1}<\cdots<i_{r}\leq k}y_{i_{1}}\cdots y_{i_{r}%
}=r!\binom{k}{r}k^{-r},
\]
and (\ref{max1}) follows from (I). Theorem \ref{MSt} is proved.
\end{proof}

\subsection{Proof of Theorem \ref{th1}}

We precede the proof by two propositions, which are not obvious for arbitrary
$p>1$.

\begin{proposition}
\label{pro1}If $G$ is a complete $k$-partite $r$-graph and $p>1$, then every
nonnegative vector to $\lambda^{\left(  p\right)  }\left(  G\right)  $ is positive.
\end{proposition}

\begin{proof}
Let $G$ be a complete $k$-partite $r$-graph, let $p>1$, and $\mathbf{x}$ be a
nonnegative eigenvector to $\lambda^{\left(  p\right)  }\left(  G\right)  .$
Assume for a contradiction that $\mathbf{x}$ has zero entries. Then, by Lemma
\ref{eqth}, all entries within the same partition sets must be equal to $0$ as
well. Let $G^{\prime}$ be the graph induced by the vertices with positive
entries in $\mathbf{x}.$ Clearly $G^{\prime}$ is complete $l$-partite, where
$r\leq l<k.$ If all parts of $G^{\prime}$ are of size $1,$ then $G^{\prime
}=K_{l}^{r}$ and $G$ contains a $K_{l+1}^{r};$ so we have
\[
\lambda^{\left(  p\right)  }\left(  G\right)  \geq\lambda^{\left(  p\right)
}\left(  K_{l+1}^{r}\right)  >\lambda^{\left(  p\right)  }\left(  K_{l}%
^{r}\right)  =\lambda^{\left(  p\right)  }\left(  G^{\prime}\right)
=\lambda^{\left(  p\right)  }\left(  G\right)  ,
\]
a contradiction. Thus $G^{\prime}$ contains a partition set of size at least
$2.$ Let $i$ belong to a partition set of size at least $2,$ and let $j$ be a
vertex such that $x_{j}$ is $0,$ i.e., $j$ does not belong to $V\left(
G^{\prime}\right)  .$ Set $x_{j}=x_{i}$ and $x_{i}=0$ and write $\mathbf{x}%
^{\prime}$ for the resulting vector. Obviously $\left\vert \mathbf{x}^{\prime
}\right\vert _{p}=1$, but we shall show that $P_{G}\left(  \mathbf{x}^{\prime
}\right)  >P_{G}\left(  \mathbf{x}\right)  ,$ which is a contradiction.
Indeed, if $\left\{  i_{1},\ldots,i_{r-1}\right\}  \subset V\left(  G^{\prime
}\right)  $ is such that $\left\{  i,i_{1},\ldots,i_{r-1}\right\}  \in
E\left(  G^{\prime}\right)  ,$ then $\left\{  j,i_{1},\ldots,i_{r-1}\right\}
\in E\left(  G\right)  .$ However, if $i^{\prime}$ belongs to the same
partition as $i,$ there is a set $\left\{  i^{\prime},i_{1},\ldots
,i_{r-2}\right\}  \subset V\left(  G^{\prime}\right)  $ such that $\left\{
j,i^{\prime},i_{1},\ldots,i_{r-2}\right\}  \in E\left(  G\right)  ,$ but
$\left\{  i,i^{\prime},i_{1},\ldots,i_{r-2}\right\}  \notin E\left(
G^{\prime}\right)  ;$ and since $x_{j}x_{i^{\prime}}x_{i_{1}}\cdots
x_{i_{r-2}}>0,$ we see that $P_{G}\left(  \mathbf{x}^{\prime}\right)
>P_{G}\left(  \mathbf{x}\right)  .$ This completes the proof of Proposition
\ref{pro1}.
\end{proof}

\begin{proposition}
\label{pro2}Let $p\geq1,$ and let $G$ be an $r$-graph such that every
nonnegative vector to $\lambda^{\left(  p\right)  }\left(  G\right)  $ is
positive. If $H$ is a subgraph of $G,$ then $\lambda^{\left(  p\right)
}\left(  H\right)  <\lambda^{\left(  p\right)  }\left(  G\right)  ,$ unless
$H=G.$
\end{proposition}

\begin{proof}
If $\lambda^{\left(  p\right)  }\left(  H\right)  =\lambda^{\left(  p\right)
}\left(  G\right)  ,$ then $V\left(  H\right)  =V\left(  G\right)  ,$
otherwise by adding zero entries, any nonnegative eigenvector to
$\lambda^{\left(  p\right)  }\left(  H\right)  $ can be extended to a
nonnegative eigenvector to $\lambda^{\left(  p\right)  }\left(  G\right)  $
that is nonpositive, contrary to the assumption. By the same token, if
$\mathbf{x}$ is an eigenvector to $H,$ it must be positive. So if $H$ has
fewer edges than $G,$ then $\lambda^{\left(  p\right)  }\left(  G\right)
=P_{H}\left(  \mathbf{x}\right)  <P_{G}\left(  \mathbf{x}\right)  ,$ a
contradiction completing the proof.
\end{proof}

\bigskip

\begin{proof}
[\textbf{Proof of Theorem }\ref{th1}]Let $G$ be a $k$-partite $r$-graph of
order $n,$ with maximum $p$-spectral radius. Proposition \ref{pro2} implies
that $G$ is complete $k$-partite; let $V_{1},\ldots,V_{k}$ be the partition
sets of $G.$ For each $i\in\left[  k\right]  ,$ set $\left\vert V_{i}%
\right\vert =n_{i}$ and suppose that $n_{1}\leq\cdots\leq n_{k}$. Assume for a
contradiction that $n_{k}-n_{1}\geq2.$ To begin with, Proposition \ref{pro1}
implies that $\mathbf{x}$\textbf{ }is a positive eigenvector to $\lambda
^{\left(  p\right)  }\left(  G\right)  $ and Lemma \ref{eqth} implies that all
entries belonging to the same partition set are equal. Thus, for each
$i\in\left[  k\right]  ,$ write $a_{i}$ for the value of the entries in
$V_{i}.$

Set $c=n_{1}a_{1}^{p}+n_{k}a_{k}^{p}$ and let
\begin{align*}
S_{1}  &  =%
%TCIMACRO{\dsum \limits_{1\text{ }<\text{ }i_{1}\text{ }<\cdots<\text{
%\ }i_{r-1}\text{ }<\text{ }k}}%
%BeginExpansion
{\displaystyle\sum\limits_{1\text{ }<\text{ }i_{1}\text{ }<\cdots<\text{
\ }i_{r-1}\text{ }<\text{ }k}}
%EndExpansion
n_{i_{1}}a_{i_{1}}\cdots n_{i_{r-1}}a_{i_{r-1}},\\
S_{2}  &  =%
%TCIMACRO{\dsum \limits_{1\text{ }<\text{ }i_{1}\text{ }<\cdots<\text{
%\ }i_{r-2}\text{ }<\text{ }k}}%
%BeginExpansion
{\displaystyle\sum\limits_{1\text{ }<\text{ }i_{1}\text{ }<\cdots<\text{
\ }i_{r-2}\text{ }<\text{ }k}}
%EndExpansion
n_{i_{1}}a_{i_{1}}\cdots n_{i_{r-2}}a_{i_{r-2}}.
\end{align*}
If $r=2,$ we let $S_{2}=1.$

Suppose first that $n_{k}+n_{1}=2l$ for some integer $l$. Let $G^{\prime}$ be
the complete $k$-partite graph with partition
\[
V(G^{\prime})=V_{1}^{\prime}\cup\cdots\cup V_{k}^{\prime},
\]
where $\left\vert V_{1}^{\prime}\right\vert =\left\vert V_{k}^{\prime
}\right\vert =l,$ $V_{1}^{\prime}\cup V_{k}^{\prime}=V_{1}\cup V_{k},$ and
$V_{i}^{\prime}=V_{i}$ for each $1<i<k.$ Now, define an $n$-vector
$\mathbf{y}$\textbf{ }which coincides with\textbf{ }$\mathbf{x}$ on $V_{2}%
\cup\cdots\cup V_{k-1}$ and for each $i\in V_{1}^{\prime}\cup V_{k}^{\prime}$
set
\[
y_{i}=\left(  c/2l\right)  ^{1/p}=b.
\]
Note first that $y_{1}^{p}+\cdots+y_{n}^{p}=1$. Further, note that
\[
P_{G^{\prime}}\left(  \mathbf{y}\right)  -P_{G}\left(  \mathbf{x}\right)
=\left(  l^{2}b^{2}-n_{1}n_{k}a_{1}a_{k}\right)  S_{2}+\left(  2lb-\left(
n_{1}a_{1}+n_{k}a_{k}\right)  \right)  S_{1}.
\]
We shall prove that
\[
l^{2}b^{2}-n_{1}n_{k}a_{1}a_{k}>0,\text{ and \ }2lb-(n_{1}a_{1}+n_{k}%
a_{k})\geq0,
\]
which implies that $P_{G^{\prime}}\left(  \mathbf{y}\right)  >P_{G}\left(
\mathbf{x}\right)  .$

Indeed, since $l^{2}>n_{1}n_{k},$ and $p>1,$ the AM-GM inequality implies
that
\begin{align*}
l^{2}b^{2}  &  =l^{2}\left(  \frac{n_{1}a_{1}^{p}+n_{k}a_{k}^{p}}{2l}\right)
^{2/p}\geq l^{2-\frac{2}{p}}\left(  \sqrt{n_{1}a_{1}^{p}n_{k}a_{k}^{p}%
}\right)  ^{2/p}=\left(  \frac{l^{2}}{n_{1}n_{k}}\right)  ^{1-1/p}n_{1}%
n_{k}a_{1}a_{k}\\
&  >n_{1}n_{k}a_{1}a_{k}.
\end{align*}
On the other hand, the PM inequality implies that
\[
2lb=2l\left(  \frac{n_{1}a_{1}^{p}+n_{k}a_{k}^{p}}{2l}\right)  ^{1/p}%
\geq2l\left(  \frac{n_{1}a_{1}+n_{k}a_{k}}{2l}\right)  =n_{1}a_{1}+n_{k}%
a_{k}.
\]
Therefore,
\[
\lambda^{\left(  p\right)  }\left(  G^{\prime}\right)  \geq P_{G^{\prime}%
}\left(  \mathbf{y}\right)  >P_{G}\left(  \mathbf{x}\right)  =\lambda^{\left(
p\right)  }\left(  G\right)  ,
\]
contradicting the choice of $G,$ and completing the proof if $n_{k}+n_{1}$ is even.

Suppose now that $n_{k}+n_{1}=2l+1$ for some integer $l.$ Let $G^{\prime}$ be
the complete $k$-partite graph with partition
\[
V(G^{\prime})=V_{1}^{\prime}\cup\cdots\cup V_{k}^{\prime},
\]
where $\left\vert V_{1}^{\prime}\right\vert =l,$ $\left\vert V_{k}^{\prime
}\right\vert =l+1,$ $V_{1}^{\prime}\cup V_{k}^{\prime}=V_{1}\cup V_{k},$ and
$V_{i}^{\prime}=V_{i}$ for each $1<i<k.$ Now, define an $n$-vector
$\mathbf{y}$\textbf{ }which coincides with\textbf{ }$\mathbf{x}$ on $V_{2}%
\cup\cdots\cup V_{k-1}$ and for each $i\in V_{1}^{\prime}\cup V_{k}^{\prime}$
set
\[
y_{i}=\left(  c/\left(  2l+1\right)  \right)  ^{1/p}=b.
\]
Note first that $y_{1}^{p}+\cdots+y_{n}^{p}=1$. Like above,%
\[
P_{G^{\prime}}\left(  \mathbf{y}\right)  -P_{G}\left(  \mathbf{x}\right)
=\left(  l\left(  l+1\right)  b^{2}-n_{1}n_{k}a_{1}a_{k}\right)  S_{2}+\left(
\left(  2l+1\right)  b-(n_{1}a_{1}+n_{k}a_{k})\right)  S_{1}.
\]
We shall prove that%
\[
\left(  2l+1\right)  b-(n_{1}a_{1}+n_{k}a_{k})\geq0,
\]
and if $p>9/8,$ then
\[
l\left(  l+1\right)  b^{2}-n_{1}n_{k}a_{1}a_{k}>0.
\]

Indeed, the first of these inequalities follows by the PM inequality as
\[
\left(  2l+1\right)  b=\left(  2l+1\right)  \left(  \frac{n_{1}a_{1}^{p}%
+n_{k}a_{k}^{p}}{2l+1}\right)  ^{1/p}\geq\left(  2l+1\right)  \left(
\frac{n_{1}a_{1}+n_{k}a_{k}}{2l+1}\right)  =n_{1}a_{1}+n_{k}a_{k}.
\]

Further, if $p>9/8,$ Bernoulli's inequality entails
\[
\left(  \frac{\left(  2l+1\right)  ^{2}}{4\left(  l-1\right)  \left(
l+2\right)  }\right)  ^{1/p}=\left(  1+\frac{9}{4\left(  l-1\right)  \left(
l+2\right)  }\right)  ^{1/p}\leq1+\frac{9}{4p\left(  l-1\right)  \left(
l+2\right)  }<\frac{l\left(  l+1\right)  }{\left(  l-1\right)  \left(
l+2\right)  }.
\]
Now, in view of $n_{1}n_{k}\leq\left(  l-1\right)  \left(  l+2\right)  ,$ the
AM-GM inequality implies that
\begin{align*}
l\left(  l+1\right)  b^{2}  &  =l\left(  l+1\right)  \left(  \frac{n_{1}%
a_{1}^{p}+n_{k}a_{k}^{p}}{2l+1}\right)  ^{2/p}\\
&  \geq l\left(  l+1\right)  \left(  \frac{1}{2l+1}\right)  ^{2/p}\left(
2\sqrt{n_{1}a_{1}^{p}n_{k}a_{k}^{p}}\right)  ^{2/p}\\
&  =l\left(  l+1\right)  \left(  \frac{4}{\left(  2l+1\right)  ^{2}}\right)
^{1/p}\left(  \frac{1}{n_{1}n_{k}}\right)  ^{1-1/p}n_{1}n_{k}a_{1}a_{k}\\
&  \geq l\left(  l+1\right)  \left(  \frac{4}{\left(  2l+1\right)  ^{2}%
}\right)  ^{\frac{1}{p}}\left(  \frac{1}{\left(  l-1\right)  \left(
l+2\right)  }\right)  ^{1-\frac{1}{p}}n_{1}n_{k}a_{1}a_{k}\\
&  =\frac{l\left(  l+1\right)  }{\left(  l-1\right)  \left(  l+2\right)
}\left(  \frac{4\left(  l-1\right)  \left(  l+2\right)  }{\left(  2l+1\right)
^{2}}\right)  ^{\frac{1}{p}}n_{1}n_{k}a_{1}a_{k}\\
&  >n_{1}n_{k}a_{1}a_{k}.
\end{align*}

In summary, if $l\left(  l+1\right)  b^{2}-n_{1}n_{k}a_{1}a_{k}>0$ or if
$p>9/8,$ we obtain a contradiction
\[
\lambda^{\left(  p\right)  }\left(  G^{\prime}\right)  \geq P_{G^{\prime}%
}\left(  \mathbf{y}\right)  >P_{G}\left(  \mathbf{x}\right)  =\lambda^{\left(
p\right)  }\left(  G\right)  .
\]
To finish the proof we shall consider the case when $p\leq9/8$ and $l\left(
l+1\right)  b^{2}-n_{1}n_{k}a_{1}a_{k}\leq0.$ Clearly, the latter inequality
can be rewritten as
\begin{equation}
a_{1}a_{k}\geq\frac{l\left(  l+1\right)  }{n_{1}n_{k}}\left(  \frac{c}%
{2l+1}\right)  ^{2/p}. \label{lo1}%
\end{equation}
Define an $n$-vector $\mathbf{z}$\textbf{ }which coincides with\textbf{
}$\mathbf{x}$ on $V_{2}\cup\cdots\cup V_{k-1},$ and for every $i\in
V_{1}^{\prime}$ and $j\in V_{k}^{\prime}$ set
\begin{align*}
z_{i}  &  =n_{1}a_{1}/l=b_{1},\text{ }\\
z_{j}  &  =n_{k}a_{k}/\left(  l+1\right)  =b_{k}.
\end{align*}
First note that
\[
z_{1}^{p}+\cdots+z_{n}^{p}=1-(n_{1}a_{1}^{p}+n_{k}a_{k}^{p})+lb_{1}%
^{p}+\left(  l+1\right)  b_{k}^{p}.
\]
We also have
\[
P_{G^{\prime}}\left(  \mathbf{z}\right)  -P_{G}\left(  \mathbf{x}\right)
=\left(  l\left(  l+1\right)  b_{1}b_{k}-n_{1}n_{k}a_{1}a_{k}\right)
S_{2}+\left(  lb_{1}+\left(  l+1\right)  b_{k}-(n_{1}a_{1}+n_{k}a_{k})\right)
S_{2}=0.
\]

Noting that%
\[
\lambda^{\left(  p\right)  }\left(  G\right)  =P_{G}\left(  \mathbf{x}\right)
=P_{G^{\prime}}\left(  \mathbf{z}\right)  \leq\lambda^{\left(  p\right)
}\left(  G^{\prime}\right)  \left\vert \mathbf{z}\right\vert _{p}^{r},
\]
in view of $\lambda^{\left(  p\right)  }\left(  G^{\prime}\right)  \leq
\lambda^{\left(  p\right)  }\left(  G\right)  ,$ we see that $|\mathbf{z}%
|_{p}\geq1.$ Hence
\[
lb_{1}^{p}+\left(  l+1\right)  b_{k}^{p}\geq n_{1}a_{1}^{p}+n_{k}a_{k}^{p},
\]
and so
\[
l\left(  \frac{n_{1}a_{1}}{l}\right)  ^{p}+\left(  l+1\right)  \left(
\frac{n_{k}a_{k}}{l+1}\right)  ^{p}\geq n_{1}a_{1}^{p}+n_{k}a_{k}^{p},
\]
implying that
\[
n_{k}a_{k}^{p}\left(  \left(  n_{k}/\left(  l+1\right)  \right)
^{p-1}-1\right)  \geq n_{1}a_{1}^{p}\left(  1-\left(  n_{1}/l\right)
^{p-1}\right)  .
\]
In view of $p>1,$ $\left(  n_{k}/\left(  l+1\right)  \right)  ^{p-1}-1>0$ and
$1-\left(  n_{1}/l\right)  ^{p-1}>0,$ we obtain
\[
\frac{n_{k}a_{k}^{p}}{n_{1}a_{1}^{p}}\geq\frac{1-\left(  n_{1}/l\right)
^{p-1}}{\left(  n_{k}/\left(  l+1\right)  \right)  ^{p-1}-1}.
\]
Now, in view of $1<p\leq9/8,$ Bernoulli's inequality gives
\[
\left(  \frac{n_{1}}{l}\right)  ^{p-1}=\left(  1-\frac{l-n_{1}}{l}\right)
^{p-1}<1-\frac{\left(  p-1\right)  \left(  l-n_{1}\right)  }{l},
\]
and so%
\[
\left(  \frac{n_{k}}{l+1}\right)  ^{p-1}=\left(  1+\frac{n_{k}-l-1}%
{l+1}\right)  ^{p-1}<1+\frac{\left(  p-1\right)  \left(  n_{k}-l-1\right)
}{l+1}.
\]
Hence, in view of $l-n_{1}=n_{k}-l-1,$ we see that
\[
\frac{n_{k}a_{k}^{p}}{n_{1}a_{1}^{p}}>\frac{l+1}{l}.
\]
Since $n_{1}a_{1}^{p}+n_{k}a_{k}^{p}=c,$ it is easy to show that
\[
n_{1}a_{1}^{p}n_{k}a_{k}^{p}<\frac{l\left(  l+1\right)  }{\left(  2l+1\right)
^{2}}c^{2},
\]
and so,
\[
a_{1}a_{k}<\left(  \frac{l\left(  l+1\right)  }{n_{1}n_{k}}\right)
^{1/p}\left(  \frac{c}{2l+1}\right)  ^{2/p}.
\]
This, together with (\ref{lo1}), implies that
\[
\left(  \frac{l\left(  l+1\right)  }{n_{1}n_{k}}\right)  ^{1/p}>\frac{l\left(
l+1\right)  }{n_{1}n_{k}},
\]
which is a contradiction, since $l\left(  l+1\right)  >n_{1}n_{k}$ and
$1/p<1.$ Theorem \ref{th1} is proved.
\end{proof}

\subsection{Proof of Theorem \ref{th2}}

\begin{proof}
Let $\left[  x_{i}\right]  $ be a nonnegative eigenvector to $\lambda^{\left(
p\right)  }\left(  G\right)  .$ The PM inequality implies that
\[
\lambda^{\left(  p\right)  }\left(  G\right)  =r!\sum_{\left\{  i_{1}%
,\ldots,i_{r}\right\}  \in E\left(  G\right)  }x_{i_{1}}\cdots x_{i_{r}}\leq
r!e\left(  G\right)  ^{1-1/p}\left(  \sum_{\left\{  i_{1},\ldots
,i_{r}\right\}  \in E\left(  G\right)  }x_{i_{1}}^{p}\cdots x_{i_{r}}%
^{p}\right)  ^{1/p}.
\]
Now, letting $\mathbf{y}=\left(  x_{1}^{p},\ldots,x_{n}^{p}\right)  ,$ Theorem
\ref{MSt} implies that
\[
\sum_{\left\{  i_{1},\ldots,i_{r}\right\}  \in E\left(  G\right)  }x_{i_{1}%
}^{p}\cdots x_{i_{r}}^{p}\leq\binom{k}{r}k^{-r}.
\]
Let $G_{2}$ be the $2$-section of $G,$ that is to say $V\left(  G_{2}\right)
=V\left(  G\right)  $ and $E\left(  G_{2}\right)  $ is the set of all
$2$-subsets of edges of $G.$ Every edge of $G$ corresponds to unique
$r$-clique in $G_{2},$ so the number of $r$-cliques $k_{r}\left(
G_{2}\right)  $ of $G_{2}$ satisfies $k_{r}\left(  G_{2}\right)  \geq e\left(
G\right)  .$ On the other hand, clearly $G_{2}$ is $k$-partite, and so it
contains no $K_{k+1}.$ By Zykov's theorem \cite{Zyk49} (see also Erd\H{o}s
\cite{Erd62}),
\begin{equation}
k_{r}\left(  G_{2}\right)  \leq k_{r}\left(  T_{k}\left(  n\right)  \right)
\leq\binom{k}{r}\left(  \frac{n}{k}\right)  ^{r}, \label{Z}%
\end{equation}
with equality holding if and only if $k|n$ and $G_{2}=T_{k}\left(  n\right)
.$ We get
\[
e\left(  G\right)  \leq\binom{k}{r}\left(  \frac{n}{k}\right)  ^{r},
\]
with equality holding if and only if $k|n$ and $G=T_{k}^{r}\left(  n\right)
.$ Therefore,
\[
\lambda^{\left(  p\right)  }\left(  G\right)  \leq r!\left(  \binom{k}%
{r}\left(  \frac{n}{k}\right)  ^{r}\right)  ^{1-1/p}\left(  \binom{k}{r}%
k^{-r}\right)  ^{1/p},
\]
implying inequality (\ref{in1}).

If equality holds in (\ref{in1}), then equality holds in (\ref{Z}), and so
$k|n$ and $G=T_{k}^{r}\left(  n\right)  .$
\end{proof}

\subsection{Proof of Theorem \ref{th3}}

For the proof of the theorem we shall need a few general statements.

\begin{proposition}
\label{eqt1}Let $r\geq3$ and $G$ be a complete $k$-chromatic $r$-graph and
$\left[  x_{i}\right]  $ be a nonnegative vector to $\lambda^{\left(
1\right)  }\left(  G\right)  .$ If the vertices $u$ and $v$ belong to the same
vertex class $U$, then $x_{u}=x_{v}$.
\end{proposition}

\begin{proof}
As in Lemma \ref{eqth} note that
\[
P_{G}\left(  \left[  x_{i}\right]  \right)  =x_{u}A+x_{v}A+x_{u}x_{v}B+C,
\]
where $A,B,C$ are independent of $x_{u}$ and $x_{v}.$ Assume that $x_{u}\neq
x_{v}$ and define a vector $\left[  x_{i}^{\prime}\right]  $ such that
\[
x_{u}^{\prime}=x_{v}^{\prime}=\frac{x_{u}+x_{v}}{2},\text{ }x_{i}^{\prime
}=x_{i}\text{ \ if \ }i\in\left[  n\right]  \backslash\left\{  u,v\right\}  .
\]
Clearly $x_{1}^{\prime}+\cdots+x_{n}^{\prime}=x_{1}+\cdots+x_{n}=1,$ while
\[
P_{G}\left(  \left[  x_{i}^{\prime}\right]  \right)  -P_{G}\left(  \left[
x_{i}\right]  \right)  =\frac{\left(  x_{u}-x_{v}\right)  ^{2}}{4}B\geq0.
\]
To complete the proof we shall show that $B>0,$ which will contradict that
$P_{G}\left(  \left[  x_{i}^{\prime}\right]  \right)  \leq P_{G}\left(
\left[  x_{i}\right]  \right)  .$ Choose an edge $\left\{  i_{1},\ldots
,i_{r}\right\}  $ with $x_{i_{1}}>0,\ldots,x_{i_{r}}>0.$ Obviously, the set
$\left\{  i_{1},\ldots,i_{r}\right\}  \backslash\left\{  u,v\right\}  $
contains a vertex not in $U,$ say $i_{r}.$ By symmetry, we can assume that
$\left\{  u,v\right\}  \cap\left\{  i_{3},\ldots,i_{r}\right\}  =\varnothing,$
and hence the set $\left\{  u,v,i_{3},\ldots,i_{r}\right\}  $ is an edge of
$G.$ Now
\[
B\geq x_{i_{3}}\cdots x_{i_{r}}>0
\]
as claimed.
\end{proof}

\begin{proposition}
\label{pro3}Let $r\geq3$ and $G$ be a complete $k$-chromatic $r$-graph. If
$p\geq1,$ then every nonnegative eigenvector to $\lambda^{\left(  p\right)
}\left(  G\right)  $ is positive.
\end{proposition}

\begin{proof}
Let $\left[  x_{i}\right]  $ be a nonnegative eigenvector to $\lambda^{\left(
p\right)  }\left(  G\right)  .$ In view of the previous proposition and Lemma
\ref{eqth}, all entries of $\left[  x_{i}\right]  $ belonging to the same
vertex class are equal, so if an entry is zero, then all entries in the same
vertex class are zero. Let $G^{\prime}$ be the the graph induced by the
vertices with positive entries in $\left[  x_{i}\right]  .$ Clearly
$G^{\prime}$ is complete $l$-chromatic, where $r\leq l<k.$ If all parts of
$G^{\prime}$ are of size at most $r-1,$ then $G^{\prime}$ is a complete graph
of order say $m.$ Then $G$ contains a complete graph of order $m+1,$ so we
have
\[
\lambda^{\left(  p\right)  }\left(  G\right)  \geq\lambda^{\left(  p\right)
}\left(  K_{m+1}^{r}\right)  >\lambda^{\left(  p\right)  }\left(  K_{m}%
^{r}\right)  =\lambda^{\left(  p\right)  }\left(  G^{\prime}\right)
=\lambda^{\left(  p\right)  }\left(  G\right)  ,
\]
a contradiction. So $G^{\prime}$ contains a partition set of size at least
$r.$ Let $i$ belong to a vertex class $U$ of size at least $r,$ and let $j$ be
a vertex such that $x_{j}$ is $0,$ that is to say, $j\notin V\left(
G^{\prime}\right)  .$ Set $x_{j}=x_{i}$ and $x_{i}=0,$ and write
$\mathbf{x}^{\prime}$ for the resulting vector. Obviously $\left\vert
\mathbf{x}^{\prime}\right\vert _{p}=1$, but we shall show that $P_{G}\left(
\mathbf{x}^{\prime}\right)  >P_{G}\left(  \mathbf{x}\right)  .$ Indeed, if
$\left\{  i_{1},\ldots,i_{r-1}\right\}  \subset V\left(  G^{\prime}\right)  $
is such that $\left\{  i,i_{1},\ldots,i_{r-1}\right\}  \in E\left(  G^{\prime
}\right)  ,$ then $\left\{  j,i_{1},\ldots,i_{r-1}\right\}  \in E\left(
G\right)  .$ However, if $\left\{  i_{1},\ldots,i_{r-1}\right\}  \subset
U\backslash\left\{  i\right\}  ,$ then $\left\{  j,i_{1},\ldots,i_{r-1}%
\right\}  \in E\left(  G\right)  ,$ but $\left\{  i,i_{1},\ldots
,i_{r-1}\right\}  \notin E\left(  G^{\prime}\right)  .$ Since $x_{i_{1}}\cdots
x_{i_{r-1}}>0,$ we see that $P_{G}\left(  \mathbf{x}^{\prime}\right)
>P_{G}\left(  \mathbf{x}\right)  .$ This contradiction completes the proof of
Proposition \ref{pro3}.
\end{proof}

\bigskip

Now we are ready to carry out the proof of Theorem \ref{th3}.

\begin{proof}
[\textbf{Proof of Theorem }\ref{th3}]Let $G$ be a $k$-chromatic $3$-graph of
order $n$ with maximum $p$-spectral radius. Propositions \ref{pro3} and
\ref{pro2} imply that $G$ is complete $k$-chromatic; let $V_{1},\ldots,V_{k}$
be the vertex sets of $G;$ for every $i\in\left[  k\right]  ,$ set $\left\vert
V_{i}\right\vert =n_{i}$ and suppose that $n_{1}\leq\cdots\leq n_{k}$. Assume
for a contradiction that $n_{k}-n_{1}\geq2.$ Proposition \ref{pro3} implies
that $\mathbf{x}$\textbf{ }is a positive eigenvector to $\lambda^{\left(
p\right)  }\left(  G\right)  ,$ and Proposition \ref{eqt1} implies that all
entries belonging to the same partition set are equal. For each $i\in\left[
k\right]  $ write $a_{i}$ for the value of the entries in $V_{i}.$ Clearly%
\[
P_{G}\left(  \mathbf{x}\right)  =\sum_{1\leq i<j\leq k}\left(  \binom{n_{i}%
}{2}n_{j}a_{i}^{2}a_{j}+\binom{n_{j}}{2}n_{i}a_{j}^{2}a_{i}\right)
+\sum_{1\leq i<j<m\leq k}n_{i}n_{j}n_{m}a_{i}a_{j}a_{m}.
\]
Set
\begin{align*}
S_{1}  &  =\sum_{1<i<j<k}\left(  \binom{n_{i}}{2}a_{i}^{2}+\binom{n_{j}}%
{2}a_{j}^{2}+n_{i}n_{j}a_{i}a_{j}\right)  ,\\
S_{2}  &  =%
%TCIMACRO{\dsum \limits_{i=2}^{k-1}}%
%BeginExpansion
{\displaystyle\sum\limits_{i=2}^{k-1}}
%EndExpansion
n_{i}a_{i}.
\end{align*}
Also, set $c=n_{1}a_{1}^{p}+n_{k}a_{k}^{p}.$

We shall exploit the following proof idea several times. We shall define a
complete $k$-chromatic graph $G^{\prime}$ with partition $V(G^{\prime}%
)=V_{1}^{\prime}\cup\cdots\cup V_{k}^{\prime},$ where $V_{1}^{\prime}\cup
V_{k}^{\prime}=V_{1}\cup V_{k},$ and $V_{i}^{\prime}=V_{i}$ for each $1<i<k.$
Thus, $G^{\prime}$ will be completely described by the numbers $m_{1}%
=\left\vert V_{1}^{\prime}\right\vert $ and $m_{k}=\left\vert V_{k}^{\prime
}\right\vert .$ Next we shall define an $n$-vector $\mathbf{y}$\textbf{ }which
coincides with\textbf{ }$\mathbf{x}$ on $V_{2}\cup\cdots\cup V_{k-1}$ and for
every $i\in V_{1}^{\prime}$ and $j\in V_{k}^{\prime}$ we shall set
$y_{i}=b_{1}$ and $y_{j}=b_{k}$, where $b_{1}$ and $b_{k\text{ }}$ are chosen
so that $m_{1}b_{1}^{p}+m_{k}b_{k}^{p}\leq c.$ Thus, $\mathbf{y}$ will be
completely described by the numbers $b_{1}$ and $b_{k\text{ }}.$ Also note
that $y_{1}^{p}+\cdots+y_{n}^{p}\leq1.$ Let us define the expressions%
\begin{align*}
P_{1}  &  =\left(  m_{1}b_{1}+m_{k}b_{k}\right)  -\left(  n_{1}a_{1}%
+n_{k}a_{k}\right)  ,\\
P_{2}  &  =\left(  \left(  m_{1}b_{1}+m_{k}b_{k}\right)  ^{2}-m_{1}b_{1}%
^{2}-m_{k}b_{k}^{2}\right)  -\left(  \left(  n_{1}a_{1}+n_{k}a_{k}\right)
^{2}-n_{1}a_{1}^{2}-n_{k}a_{k}^{2}\right)  ,\\
P_{3}  &  =\left(  \binom{m_{1}}{2}m_{k}b_{1}^{2}b_{k}+\binom{m_{k}}{2}%
m_{1}b_{k}^{2}b_{1}\right)  -\left(  \binom{n_{1}}{2}n_{k}a_{1}^{2}%
a_{k}+\binom{n_{k}}{2}n_{1}a_{k}^{2}a_{1}\right)  ,
\end{align*}
and note that
\[
P_{G^{\prime}}\left(  \mathbf{y}\right)  -P_{G}\left(  \mathbf{x}\right)
=P_{1}S_{1}+\frac{1}{2}P_{2}S_{2}+P_{3}.
\]
After choosing $m_{1},m_{k},$ $b_{1}$ and $b_{k},$ we shall show that
$P_{1}\geq0,$ $P_{2}\geq0$ and $P_{3}>0,$ which contradicts the choice of $G$.

First, suppose that $n_{k}+n_{1}=2l$ for some integer $l$. In this case let
\[
m_{1}=m_{k}=l,\text{ }b_{1}=b_{k}=\left(  \frac{c}{2l}\right)  ^{1/p}=b.
\]
Note first that $P_{1}\geq0$ follows by the PM inequality
\begin{equation}
n_{1}a_{1}+n_{k}a_{k}\leq2l\left(  \frac{n_{1}}{2l}a_{1}^{p}+\frac{n_{k}}%
{2l}a_{k}^{p}\right)  ^{1/p}=2lb. \label{mi}%
\end{equation}
To prove $P_{2}\geq0$ note that
\begin{align*}
\left(  n_{1}a_{1}+n_{k}a_{k}\right)  ^{2}-n_{1}a_{1}^{2}-n_{k}a_{k}^{2}  &
\leq\left(  n_{1}a_{1}+n_{k}a_{k}\right)  ^{2}-\frac{1}{2l}\left(  n_{1}%
a_{1}+n_{k}a_{k}\right)  ^{2}\\
&  =2l\left(  2l-1\right)  \left(  \frac{n_{1}}{2l}a_{1}+\frac{n_{k}}{2l}%
a_{k}\right)  ^{2}\\
&  \leq2l\left(  2l-1\right)  \left(  \frac{n_{1}}{2l}a_{1}^{p}+\frac{n_{k}%
}{2l}a_{k}^{p}\right)  ^{2/p}=2l\left(  2l-1\right)  b^{2}.
\end{align*}
Finally, we shall prove that $P_{3}>0.$ Let us start with the observation
\begin{align*}
\binom{n_{1}}{2}n_{k}a_{1}^{2}a_{k}+\binom{n_{k}}{2}n_{1}a_{k}^{2}a_{1}  &
=\frac{1}{2}n_{1}n_{k}a_{1}a_{k}\left(  n_{1}a_{1}+n_{k}a_{k}-a_{1}%
-a_{k}\right) \\
&  \leq\frac{1}{2}n_{1}n_{k}a_{1}a_{k}\left(  n_{1}a_{1}+n_{k}a_{k}%
-2\sqrt{a_{1}a_{k}}\right) \\
&  \leq n_{1}n_{k}a_{1}a_{k}\left(  lb-\sqrt{a_{1}a_{k}}\right)  .
\end{align*}
Now, since $z\left(  lb-\sqrt{z}\right)  $ is increasing for$\sqrt{z}\leq$
$2lb/3$ and
\[
\sqrt{a_{1}a_{k}}=\frac{1}{\sqrt{n_{1}n_{k}}}\sqrt{n_{1}a_{1}n_{k}a_{k}}%
\leq\frac{lb}{\sqrt{n_{1}n_{k}}}\leq\frac{lb}{\sqrt{3}}<\frac{2}{3}lb,
\]
we see that
\begin{align*}
\binom{n_{1}}{2}n_{k}a_{1}^{2}a_{k}+\binom{n_{k}}{2}n_{1}a_{k}^{2}a_{1}  &
\leq n_{1}n_{k}\frac{l^{2}b^{2}}{n_{1}n_{k}}\left(  lb-\frac{lb}{\sqrt
{n_{1}n_{k}}}\right) \\
&  <l^{2}\left(  l-1\right)  b^{3},
\end{align*}
completing the proof of $P_{3}>0$. Hence, $P_{G^{\prime}}\left(
\mathbf{y}\right)  >P_{G}\left(  \mathbf{x}\right)  ,$ contrary to the choice
of $G.$ This proves the theorem if $n_{k}+n_{1}=2l$ for some integer $l$.

Assume now that $n_{k}+n_{1}=2l+1$ for some integer $l.$ This is a more
difficult task, so we shall split the remaining part of the proof into two
cases (A) and (B) as follows:%
\begin{align*}
\text{(A) \ }p  &  >2\text{ or }\left(  l+1\right)  n_{1}a_{1}^{p}\leq
ln_{k}a_{k}^{p};\text{ }\\
\text{(B)\ \ }p  &  \leq2\text{ and }\left(  l+1\right)  n_{1}a_{1}^{p}%
>ln_{k}a_{k}^{p}.
\end{align*}

We start with (A), so assume that $p>2$ or $\left(  l+1\right)  n_{1}a_{1}%
^{p}\leq ln_{k}a_{k}^{p}$. Define $G^{\prime}$ and $\mathbf{y}$ by
\[
m_{1}=l,\text{ }m_{k}=l+1,\text{ }b_{1}=b_{k}=\left(  c/\left(  2l+1\right)
\right)  ^{1/p}=b.
\]
Again we shall prove that $P_{1}\geq0,$ $P_{2}\geq0$ and $P_{3}>0.$ First, the
PM\ inequality implies that
\[
n_{1}a_{1}+n_{k}a_{k}\leq\left(  2l+1\right)  \left(  \frac{n_{1}}{2l+1}%
a_{1}^{p}+\frac{n_{k}}{2l+1}a_{k}^{p}\right)  ^{1/p}=\left(  2l+1\right)  b,
\]
so $P_{1}\geq0.$ Next
\begin{align*}
\left(  n_{1}a_{1}+n_{k}a_{k}\right)  ^{2}-n_{1}a_{1}^{2}-n_{k}a_{k}^{2}  &
\leq\left(  n_{1}a_{1}+n_{k}a_{k}\right)  ^{2}-\frac{1}{2l+1}\left(
n_{1}a_{1}+n_{k}a_{k}\right)  ^{2}\\
&  =2l\left(  2l+1\right)  \left(  \frac{n_{1}}{2l+1}a_{1}+\frac{n_{k}}%
{2l+1}a_{k}\right)  ^{2}\\
&  \leq2l\left(  2l+1\right)  b^{2}.
\end{align*}
so $P_{2}\geq0.$ To prove that $P_{3}>0$, note that
\begin{align}
\binom{n_{1}}{2}n_{k}a_{1}^{2}a_{k}+\binom{n_{k}}{2}n_{1}a_{k}^{2}a_{1}  &
=\frac{1}{2}n_{1}n_{k}a_{1}a_{k}\left(  n_{1}a_{1}+n_{k}a_{k}-a_{1}%
-a_{k}\right) \nonumber\\
&  \leq\frac{1}{2}n_{1}n_{k}a_{1}a_{k}\left(  n_{1}a_{1}+n_{k}a_{k}%
-2\sqrt{a_{1}a_{k}}\right) \nonumber\\
&  \leq n_{1}n_{k}a_{1}a_{k}\left(  lb-\sqrt{a_{1}a_{k}}\right)  . \label{ex1}%
\end{align}
Our goal now is to bound from above the right side of (\ref{ex1}). Since the
expression $z\left(  lb-\sqrt{z}\right)  $ is increasing for$\sqrt{z}\leq$
$2lb/3,$ we focus on an upper bound on $\sqrt{a_{1}a_{k}}.$ To this end recall
that the condition of case (A) is a disjunction of two clauses. If the second
one is true, i.e., if $\left(  l+1\right)  n_{1}a_{1}^{p}\leq ln_{k}a_{k}%
^{p},$ then%
\[
n_{1}a_{1}^{p}\leq\frac{l}{2l+1}c\text{ \ and \ }n_{k}a_{k}^{p}\geq\frac
{l+1}{2l+1}c,
\]
and we find that%
\begin{align*}
n_{1}a_{1}^{p}n_{k}a_{k}^{p}  &  =\left(  \frac{c}{2}-\left(  \frac{c}%
{2}-n_{1}a_{1}^{p}\right)  \right)  \left(  \frac{c}{2}+\left(  \frac{c}%
{2}-n_{1}a_{1}^{p}\right)  \right)  =\frac{c^{2}}{4}-\left(  \frac{c}{2}%
-n_{1}a_{1}^{p}\right)  ^{2}\\
&  \leq\frac{c^{2}}{4}-\frac{c^{2}}{4\left(  2l+1\right)  ^{2}}=\frac{l\left(
l+1\right)  }{\left(  2l+1\right)  ^{2}}c^{2}.
\end{align*}
Hence,
\begin{equation}
\sqrt{a_{1}a_{k}}\leq\left(  \sqrt{n_{1}n_{k}}\right)  ^{-1/p}\left(
\frac{\sqrt{l\left(  l+1\right)  }}{2l+1}\right)  ^{1/p}c^{1/p}=\left(
\frac{l\left(  l+1\right)  }{n_{1}n_{k}}\right)  ^{1/2p}b.\nonumber
\end{equation}
Also, in view of $n_{1}n_{k}\geq4$ and $l\geq2$, we see that
\[
\sqrt{a_{1}a_{k}}\leq\left(  \frac{l\left(  l+1\right)  }{n_{1}n_{k}}\right)
^{1/2p}b<\left(  \frac{l\left(  l+1\right)  }{1\cdot4}\right)  ^{1/2}%
b\leq\left(  \frac{3l^{2}}{8}\right)  ^{1/2}b<\frac{2lb}{3}.
\]
Hence, to bound the right side of (\ref{ex1}) we replace$\sqrt{a_{1}a_{k}}$ by
$\left(  \frac{l\left(  l+1\right)  }{n_{1}n_{k}}\right)  ^{1/2p}b,$ thus
obtaining
\begin{align*}
\binom{n_{1}}{2}n_{k}a_{1}^{2}a_{k}+\binom{n_{k}}{2}n_{1}a_{k}^{2}a_{1}  &
\leq n_{1}n_{k}\left(  \frac{l\left(  l+1\right)  }{n_{1}n_{k}}\right)
^{1/p}b^{2}\left(  \frac{2l+1}{2}b-\left(  \frac{l\left(  l+1\right)  }%
{n_{1}n_{k}}\right)  ^{1/2p}b\right) \\
&  <l\left(  l+1\right)  b^{2}\left(  \frac{2l+1}{2}b-b\right)  =\frac
{l\left(  l+1\right)  \left(  2l-1\right)  }{2}b^{3}.
\end{align*}
Hence $P_{3}>0$ and $P_{G^{\prime}}\left(  \mathbf{y}\right)  >P_{G}\left(
\mathbf{x}\right)  ,$ contrary to the choice of $G.$ This completes the proof
of (A) if $\left(  l+1\right)  n_{1}a_{1}^{p}\leq ln_{k}a_{k}^{p}.$ To finish
the proof in case (A), assume that $p>2.$ Then%
\[
\sqrt{4n_{1}a_{1}^{p}n_{k}a_{k}^{p}}\leq c=\left(  2l+1\right)  b^{p},
\]
and we find that%
\begin{equation}
\sqrt{a_{1}a_{k}}\leq\left(  \frac{2l+1}{\sqrt{4n_{1}n_{k}}}\right)  ^{1/p}b.
\label{in3}%
\end{equation}
Also, in view of $n_{1}n_{k}\geq4$ and $l\geq2,$ we see\ that
\[
\sqrt{a_{1}a_{k}}\leq\left(  \frac{2l+1}{\sqrt{4n_{1}n_{k}}}\right)
^{1/p}b\leq\left(  \frac{2l+1}{\sqrt{4\cdot4}}\right)  ^{1/p}b<\left(
\frac{2l+1}{4}\right)  ^{1/2}b<\frac{2lb}{3}.
\]
Hence, to bound the right side of (\ref{ex1}) we replace$\sqrt{a_{1}a_{k}}$ by
$\left(  \frac{2l+1}{\sqrt{4n_{1}n_{k}}}\right)  ^{1/p}b,$ thus obtaining
\begin{align*}
\binom{n_{1}}{2}n_{k}a_{1}^{2}a_{k}+\binom{n_{k}}{2}n_{1}a_{k}^{2}a_{1}  &
\leq\frac{n_{1}n_{k}}{2}\left(  \frac{\left(  2l+1\right)  ^{2}}{4n_{1}n_{k}%
}\right)  ^{1/p}b^{2}\left(  \left(  2l+1\right)  b-2\left(  \frac{2l+1}%
{\sqrt{4n_{1}n_{k}}}\right)  ^{1/p}b\right) \\
&  <\frac{\left(  n_{1}n_{k}\right)  ^{1-1/p}}{2}\left(  \frac{\left(
2l+1\right)  ^{2}}{4}\right)  ^{1/p}\left(  2l-1\right)  b^{3}.
\end{align*}
To complete the proof we shall show that
\[
\left(  n_{1}n_{k}\right)  ^{1-1/p}\left(  \frac{\left(  2l+1\right)  ^{2}}%
{4}\right)  ^{1/p}\leq l\left(  l+1\right)  .
\]
Assume that this fails, that is to say,%
\[
\left(  \frac{\left(  2l+1\right)  ^{2}}{4l\left(  l+1\right)  }\right)
^{1/p}>\left(  \frac{l\left(  l+1\right)  }{\left(  n_{1}n_{k}\right)
}\right)  ^{1-1/p}.
\]
Hence, we find that%
\[
1+\frac{1}{4l\left(  l+1\right)  }=\frac{\left(  2l+1\right)  ^{2}}{4l\left(
l+1\right)  }>\left(  \frac{l\left(  l+1\right)  }{\left(  n_{1}n_{k}\right)
}\right)  ^{p-1}\geq\frac{l\left(  l+1\right)  }{\left(  l-1\right)  \left(
l+2\right)  }=1+\frac{2}{\left(  l-1\right)  \left(  l+2\right)  },
\]
which is false for $l\geq2.$ Hence, $P_{G^{\prime}}\left(  \mathbf{y}\right)
>P_{G}\left(  \mathbf{x}\right)  ,$ contrary to the choice of $G.$ This
completes the proof of case (A).

Now consider case (B). This time, define $G^{\prime}$ and $\mathbf{y}$ by%
\[
m_{1}=l,\text{ }m_{k}=l+1,\text{ }b_{1}=\frac{n_{1}a_{1}}{l},\text{ }%
b_{k}=\frac{n_{k}a_{k}}{l+1}.
\]
Our first goal is to show that
\[
lb_{1}^{p}+\left(  l+1\right)  b_{k}^{p}\leq n_{1}a_{1}^{p}+n_{k}a_{k}^{p},
\]
that is to say%
\begin{equation}
\frac{n_{1}^{p}a_{1}^{p}}{l^{p-1}}+\frac{n_{k}^{p}a_{k}^{p}}{\left(
l+1\right)  ^{p-1}}\leq n_{1}a_{1}^{p}+n_{k}a_{k}^{p}. \label{n}%
\end{equation}
Assume for a contradiction that
\[
\frac{n_{k}^{p}a_{k}^{p}}{\left(  l+1\right)  ^{p-1}}-n_{k}a_{k}^{p}%
>n_{1}a_{1}^{p}-\frac{n_{1}^{p}a_{1}^{p}}{l^{p-1}},
\]
implying that
\[
\frac{n_{1}a_{1}^{p}}{n_{k}a_{k}^{p}}<\frac{\left(  \frac{n_{k}}{l+1}\right)
^{p-1}-1}{1-\left(  \frac{n_{1}}{l}\right)  ^{p-1}}.
\]
Set $s=n_{k}-l-1=l-n_{1}$ and note that $p-1\leq1.$ Then, Bernoulli's
inequality implies that%
\[
\frac{n_{1}a_{1}^{p}}{n_{k}a_{k}^{p}}<\frac{\left(  \frac{n_{k}}{l+1}\right)
^{p-1}-1}{1-\left(  \frac{n_{1}}{l}\right)  ^{p-1}}=\frac{\left(  1+\frac
{s}{l+1}\right)  ^{p-1}-1}{1-\left(  1-\frac{s}{l}\right)  ^{p-1}}\leq
\frac{\frac{s}{l+1}}{\frac{s}{l}}=\frac{l}{l+1},
\]
contrary to the assumption of case (B). Therefore, (\ref{n}) holds.

Next, obviously $P_{1}=0;$ we shall show that $P_{2}\geq0$ and $P_{3}>0.$
Indeed,%
\[
\left(  n_{1}a_{1}+n_{k}a_{k}\right)  ^{2}-n_{1}a_{1}^{2}-n_{k}a_{k}%
^{2}=\left(  lb_{1}+\left(  l+1\right)  b_{k}\right)  ^{2}-n_{1}a_{1}%
^{2}-n_{k}a_{k}^{2},
\]
so to prove that $P_{2}\geq0$ it is enough to show that
\[
n_{1}a_{1}^{2}+n_{k}a_{k}^{2}\geq lb_{1}^{2}+\left(  l+1\right)  b_{k}%
^{2}=\frac{n_{1}^{2}a_{1}^{2}}{l}+\frac{n_{k}^{2}a_{k}^{2}}{l+1}.
\]
If this inequality failed, then%
\[
n_{1}a_{1}^{2}-\frac{n_{1}^{2}a_{1}^{2}}{l}<\frac{n_{k}^{2}a_{k}^{2}}%
{l+1}-n_{k}a_{k}^{2},
\]
and, in view of $n_{k}l>\left(  l+1\right)  n_{1}$ and $1/2\leq1/p,$ we
obtain
\[
\frac{a_{1}}{a_{k}}<\left(  \frac{n_{k}l}{\left(  l+1\right)  n_{1}}\right)
^{1/2}\leq\left(  \frac{n_{k}l}{\left(  l+1\right)  n_{1}}\right)  ^{1/p},
\]
contrary to the assumption of case (B). Therefore, $P_{2}\geq0.$

Finally, to prove $P_{3}>0,$ note that $n_{1}n_{k}a_{1}a_{k}=l\left(
l+1\right)  b_{1}b_{k}$ and we only need to prove that
\begin{align*}
\left(  n_{1}-1\right)  a_{1}+\left(  n_{k}-1\right)  a_{k}  &  <\left(
l-1\right)  b_{1}+lb_{k}\\
&  =\left(  l-1\right)  \frac{n_{1}a_{1}}{l}+l\frac{n_{k}a_{k}}{l+1}\\
&  =n_{1}a_{1}-\frac{n_{1}a_{1}}{l}+n_{k}a_{k}-\frac{n_{k}a_{k}}{l+1}.
\end{align*}
Indeed if the latter were false, we would have
\[
\frac{n_{k}-l-1}{l+1}a_{k}=\frac{n_{k}a_{k}}{l+1}-a_{k}\geq a_{1}-\frac
{n_{1}a_{1}}{l}=\frac{l-n_{1}}{l}a_{1},
\]
and so $a_{1}/a_{k}\leq l/\left(  l+1\right)  ,$ contrary to the assumption of
(B). Therefore, $P_{3}>0$ and $P_{G^{\prime}}\left(  \mathbf{y}\right)
>P_{G}\left(  \mathbf{x}\right)  ,$ contrary to the choice of $G.$ This
completes the proof of case (B).

In summary, the assumption $\left\vert n_{1}-n_{k}\right\vert \geq2$
contradicts that $G$ has maximum $p$-spectral radius among the $k$-chromatic
$3$-graphs of order $n$. Theorem \ref{th3} is proved.
\end{proof}

\subsection{Proof of theorem \ref{th4}}

It is convenient to prove a more abstract statement first.

\begin{theorem}
\label{th4.1} Let $\left[  n_{i}\right]  $ be a real $k$-vector such that
$n_{i}\geq1,$ $i=1,\ldots,k$ and $n_{1}+\cdots+n_{k}=n\geq k$ and let $\left[
a_{i}\right]  $ be a nonnegative $k$-vector with $n_{1}a_{1}+\cdots+n_{k}%
a_{k}=s$. Then the function
\[
R\left(  \left[  n_{i}\right]  ,\left[  a_{i}\right]  \right)  =\sum_{1\leq
i<j\leq k}\left(  \binom{n_{i}}{2}n_{j}a_{i}^{2}a_{j}+\binom{n_{j}}{2}%
n_{i}a_{j}^{2}a_{i}\right)  +\sum_{1\leq i<j<m\leq k}n_{i}n_{j}n_{m}a_{i}%
a_{j}a_{m}.
\]
satisfies
\[
R\left(  \left[  n_{i}\right]  ,\left[  a_{i}\right]  \right)  \leq\left(
\binom{n}{3}-k\binom{n/k}{3}\right)  \frac{s^{3}}{n^{3}},
\]
with equality holding if and only if $n_{1}=\cdots=n_{k}$ and $a_{1}%
=\cdots=a_{k}.$
\end{theorem}

To simplify the proof of Theorem \ref{th4.1} we prove an auxiliary statement first.

\begin{proposition}
\label{pro} If $n_{1}=\cdots=n_{k}=l\geq1,$ and the $k$-vector $\left[
a_{i}\right]  $ satisfies $a_{1}+\cdots+a_{k}=t,$ then%
\[
R\left(  \left[  n_{i}\right]  ,\left[  a_{i}\right]  \right)  <\left(
\binom{kl}{3}-k\binom{l}{3}\right)  \frac{t^{3}}{k^{3}},
\]
unless $a_{1}=\cdots=a_{k}.$
\end{proposition}

\begin{proof}
Since $R\left(  \left[  n_{i}\right]  ,\left[  a_{i}\right]  \right)  $ is
homogenous of degree $3$ with respect to $\left[  a_{i}\right]  $, without
loss of generality we may assume that $t=1.$ Also, let us note that
\begin{align*}
\left(  \binom{kl}{3}-k\binom{l}{3}\right)  \frac{1}{k^{3}}  &  =\frac{l}%
{6}\cdot\frac{\left(  lk-1\right)  \left(  lk-2\right)  -\left(  l-1\right)
\left(  l-2\right)  }{k^{2}}\\
&  =\frac{l^{2}}{6}\left(  l\left(  1-\frac{1}{k^{2}}\right)  -\frac{3\left(
k-1\right)  }{k^{2}}\right)  ,
\end{align*}
and furthermore if $a_{1}=\cdots=a_{k},$ then indeed
\[
R\left(  \left[  n_{i}\right]  ,\left[  a_{i}\right]  \right)  =\frac{l^{2}%
}{6}\cdot k\cdot\frac{1}{k}\left(  1-\frac{1}{k}\right)  \left(  \left(
l-3\right)  \frac{1}{k}+l\right)  =\left(  \binom{kl}{3}-k\binom{l}{3}\right)
\frac{1}{k^{3}}.
\]
We shall prove the proposition by induction on $k.$ For $k=2,$ the AM-GM
inequality implies that
\[
R\left(  \left[  n_{i}\right]  ,\left[  a_{i}\right]  \right)  =\frac
{l^{2}\left(  l-1\right)  }{2}a_{1}a_{2}\left(  a_{1}+a_{2}\right)  \leq
\frac{l^{2}\left(  l-1\right)  }{8}=\frac{l^{2}}{6}\left(  l\left(  1-\frac
{1}{4}\right)  -\frac{3}{4}\right)  ,
\]
so the assertion holds. Assume that $k>2$ and that the assertion holds for all
$k^{\prime}<k.$ Let $R\left(  \left[  n_{i}\right]  ,\left[  a_{i}\right]
\right)  $ attain maximum for some vector $\left[  a_{i}\right]  $ with
$a_{1}+\cdots+a_{k}=1.$ If $a_{i}=0$ for some $i\in\left[  k\right]  ,$ then,
by the induction assumption, and in view of $l\geq1$ and $k\geq3,$
\[
R\left(  \left[  n_{i}\right]  ,\left[  a_{i}\right]  \right)  \leq\frac
{l^{2}}{6}\left(  l\left(  1-\frac{1}{\left(  k-1\right)  ^{2}}\right)
-\frac{3\left(  k-2\right)  }{\left(  k-1\right)  ^{2}}\right)  <\frac{l^{2}%
}{6}\left(  l\left(  1-\frac{1}{k^{2}}\right)  -\frac{3\left(  k-1\right)
}{k^{2}}\right)  ,
\]
proving the assertion. So we shall assume that $a_{i}>0$ for all $i\in\left[
k\right]  .$

Our next task is to rewrite $R\left(  \left[  n_{i}\right]  ,\left[
a_{i}\right]  \right)  $ in a more convenient form. Note that
\begin{align*}
R\left(  \left[  n_{i}\right]  ,\left[  a_{i}\right]  \right)   &
=\frac{l^{2}\left(  l-1\right)  }{2}\sum_{1\leq i<j\leq k}\left(  a_{i}%
^{2}a_{j}+a_{j}^{2}a_{i}\right)  +l^{3}\sum_{1\leq i<j<m\leq k}a_{i}a_{j}%
a_{m}\\
&  =\frac{l^{2}}{2}\left(  \left(  l-1\right)  \sum_{i=1}^{k}a_{i}^{2}\left(
1-a_{i}\right)  +2l\sum_{1\leq i<j<m\leq k}a_{i}a_{j}a_{m}\right)  .
\end{align*}
Also, it is not hard to see that%
\begin{align*}
2l\sum_{1\leq i<j<m\leq k}a_{i}a_{j}a_{m}  &  =\frac{2l}{3}\sum_{1\leq i<j\leq
k}a_{i}a_{j}\left(  1-a_{i}-a_{j}\right) \\
&  =\frac{2l}{3}\sum_{1\leq i<j\leq k}a_{i}a_{j}-\frac{2l}{3}\sum_{1\leq
i<j\leq k}a_{i}a_{j}\left(  a_{i}+a_{j}\right) \\
&  =\frac{l}{3}\sum_{i=1}^{k}a_{i}\left(  1-a_{i}\right)  -\frac{2l}{3}%
\sum_{i=1}^{k}a_{i}^{2}\left(  1-a_{i}\right) \\
&  =\sum_{i=1}^{k}a_{i}\left(  1-a_{i}\right)  \left(  \frac{l}{3}-\frac
{2l}{3}a_{i}\right)  .
\end{align*}
Therefore,
\begin{align*}
R\left(  \left[  n_{i}\right]  ,\left[  a_{i}\right]  \right)   &
=\frac{l^{2}}{2}\left(  \sum_{i=1}^{k}\left(  l-1\right)  a_{i}^{2}\left(
1-a_{i}\right)  +\sum_{i=1}^{k}a_{i}\left(  1-a_{i}\right)  \left(  \frac
{l}{3}-\frac{2l}{3}a_{i}\right)  \right) \\
&  =\frac{l^{2}}{6}\sum_{i=1}^{k}a_{i}\left(  1-a_{i}\right)  \left(  \left(
l-3\right)  a_{i}+l\right) \\
&  =\frac{l^{2}}{6}\sum_{i=1}^{k}f\left(  a_{i}\right)  ,
\end{align*}
where
\[
f\left(  x\right)  =x\left(  1-x\right)  \left(  \left(  l-3\right)
x+l\right)  .
\]

The second derivative of $f\left(  x\right)  $ is $6\left(  3-l\right)  x-6;$
and so if $l\geq2,$ then $f^{\prime\prime}\left(  x\right)  <0$ for $0<x<1;$
that is to say, if $l\geq2,$ then $f\left(  x\right)  $ is concave for
$0<x<1;$ and hence $R\left(  \left[  n_{i}\right]  ,\left[  a_{i}\right]
\right)  $ attaines maximum if and only if $a_{1}=\cdots=a_{k}.$

Fianlly, let $l<2.$ By the Lagrange method, there exists $\lambda$ such that
$f^{\prime}\left(  a_{i}\right)  =\lambda$ for all $i\in\left[  k\right]  .$
Hence, if $a_{i}\neq a_{j},$ then $a_{i}$ and $a_{j}$ are the two roots of the
quadratic equation
\[
l-6x-3\left(  l-3\right)  x^{2}=\lambda,
\]
and so%
\[
a_{i}+a_{j}=-\frac{-6}{-3\left(  l-3\right)  }=\frac{2}{3-l}>1,
\]
a contradiction, showing that $a_{1}=\cdots=a_{k}.$ The induction step is
completed and Proposition \ref{pro} is proved.
\end{proof}

\bigskip

\begin{proof}
[\textbf{Proof of Theorem \ref{th4.1}}]Since $R$ is continuous in each
variable $n_{1},\ldots,n_{k},a_{1},\ldots,a_{k}$, and its domain is compact,
it attains a maximum for some $n_{1},\ldots,n_{k}$ and $a_{1},\ldots,a_{k}.$
First we shall prove the statement under the assumption that all $a_{1}%
,\ldots,a_{k}$ are positive. Also, by symmetry, we assume that $n_{1}%
\leq\cdots\leq n_{k\text{ }}.$ Our proof is by contradiction, so assume that
$n_{1}<n_{k}$ and set
\begin{align*}
S_{1}  &  =\sum_{1<i<j<k}\left(  \binom{n_{i}}{2}a_{i}^{2}+\binom{n_{j}}%
{2}a_{j}^{2}+n_{i}n_{j}a_{i}a_{j}\right)  ,\\
S_{2}  &  =%
%TCIMACRO{\dsum \limits_{i=2}^{k-1}}%
%BeginExpansion
{\displaystyle\sum\limits_{i=2}^{k-1}}
%EndExpansion
n_{i}a_{i},\\
c  &  =n_{1}a_{1}+n_{k}a_{k},\\
l  &  =\left(  n_{1}+n_{k}\right)  /2,
\end{align*}
The proof proceeds along the following line: we replace $\left[  n_{i}\right]
$ and $\left[  a_{i}\right]  $ by two $k$-vectors $\left[  m_{i}\right]  $ and
$\left[  b_{i}\right]  $ satisfying%
\begin{align*}
m_{1}  &  =m_{k}=l,\text{ \ }m_{i}=n_{i},\text{ \ }1<i<k\text{,}\\
b_{1}  &  =b_{k}=\frac{c}{2l},\text{ }b_{i}=a_{i},\text{ \ }1<i<k.
\end{align*}
Clearly $\left[  m_{i}\right]  $ and $\left[  b_{i}\right]  $ satisfy the
conditions for $\left[  n_{i}\right]  $ and $\left[  a_{i}\right]  ,$ but we
shall show that
\[
R\left(  \left[  m_{i}\right]  ,\left[  b_{i}\right]  \right)  >R\left(
\left[  n_{i}\right]  ,\left[  a_{i}\right]  \right)  ,
\]
which is a contradiction, due to the assumption that $n_{1}<n_{k}.$

To carry out this strategy, let
\begin{align*}
R_{1}  &  =\left(  lb_{1}+lb_{k}\right)  -\left(  n_{1}a_{1}+n_{k}%
a_{k}\right)  ,\\
R_{2}  &  =\left(  \left(  lb_{1}+lb_{k}\right)  ^{2}-lb_{1}^{2}-lb_{k}%
^{2}\right)  -\left(  \left(  n_{1}a_{1}+n_{k}a_{k}\right)  ^{2}-n_{1}%
a_{1}^{2}-n_{k}a_{k}^{2}\right)  ,\\
R_{3}  &  =\left(  \binom{l}{2}lb_{1}^{2}b_{k}+\binom{l}{2}lb_{k}^{2}%
b_{1}\right)  -\left(  \binom{n_{1}}{2}n_{k}a_{1}^{2}a_{k}+\binom{n_{k}}%
{2}n_{1}a_{k}^{2}a_{1}\right)  .
\end{align*}
and note that
\[
R\left(  \left[  m_{i}\right]  ,\left[  b_{i}\right]  \right)  -R\left(
\left[  n_{i}\right]  ,\left[  a_{i}\right]  \right)  =R_{1}S_{1}+\frac{1}%
{2}R_{2}S_{2}+R_{3}.
\]

Obviously $R_{1}=0;$ we shall prove that $R_{2}\geq0$ and $R_{3}>0.$ Indeed,
\[
n_{1}a_{1}^{2}+n_{k}a_{k}^{2}\geq2l\left(  \frac{n_{1}}{2l}a_{1}+\frac{n_{k}%
}{2l}a_{k}\right)  ^{2}=lb_{1}^{2}+lb_{k}^{2},
\]
and this together with $R_{1}=0,$ implies that $R_{2}\geq0.$ Finally note
that
\[
a_{1}a_{k}\leq\frac{c^{2}}{4n_{1}n_{k}}%
\]
and so,
\begin{align*}
n_{1}n_{k}a_{1}a_{k}\left(  n_{1}a_{1}+n_{k}a_{k}-a_{1}-a_{k}\right)   &  \leq
n_{1}n_{k}a_{1}a_{k}\left(  c-2\sqrt{a_{1}a_{k}}\right) \\
&  \leq\frac{c^{2}}{4}\left(  c-\frac{c}{\sqrt{n_{1}n_{k}}}\right)
<\frac{c^{3}}{4}\left(  1-\frac{1}{l}\right) \\
&  =l^{2}b_{1}b_{k}\left(  l-1\right)  \left(  b_{1}+b_{k}\right)  .
\end{align*}
Hence, $R\left(  \left[  m_{i}\right]  ,\left[  b_{i}\right]  \right)
>R\left(  \left[  n_{i}\right]  ,\left[  a_{i}\right]  \right)  ,$ a
contradiction showing that $n_{1}=\cdots=n_{k}.$ This completes the proof if
all $a_{1},\ldots,a_{k}$ are positive.

In the general case assume that $a_{1}>0,\ldots,a_{s}>0,$ $a_{s+1}%
=\cdots=a_{k}=0.$ By what we already proved, we have $n_{1}=\cdots=n_{s},$ and
Proposition \ref{pro} implies that
\begin{align*}
R\left(  \left[  n_{i}\right]  ,\left[  a_{i}\right]  \right)   &  \leq
\binom{m}{3}\frac{1}{m^{3}}-s\binom{m/s}{3}\frac{1}{m^{3}}=\frac{\left(
s-1\right)  }{6s^{2}}\left(  \left(  s+1\right)  -\frac{3s}{m}\right) \\
&  \leq\binom{n}{3}\frac{1}{n^{3}}-s\binom{n/s}{3}\frac{1}{n^{3}}<\binom{n}%
{3}\frac{1}{n^{3}}-k\binom{n/k}{3}\frac{1}{n^{3}}.
\end{align*}
An easy inspection shows that equality holds if and only if $s=k.$ Theorem
\ref{th4.1} is proved.
\end{proof}

\bigskip

\begin{proof}
[\textbf{Proof of Theorem \ref{th4}}]Since (I) is a particular case of
Proposition \ref{pth4}, we proceed directly with (II). For $p=1$ the assertion
follows from Theorem \ref{th4.1}. Let $p>1$ and let $\left[  x_{i}\right]  $
be a positive eigenvector to $\lambda^{\left(  p\right)  }\left(  G\right)  .$
The PM inequality implies that
\begin{equation}
\lambda^{\left(  p\right)  }\left(  G\right)  =r!\sum_{\left\{  i_{1}%
,\ldots,i_{r}\right\}  \in E\left(  G\right)  }x_{i_{1}}\cdots x_{i_{r}}\leq
r!e\left(  G\right)  ^{1-1/p}\left(  \sum_{\left\{  i_{1},\ldots
,i_{r}\right\}  \in E\left(  G\right)  }x_{i_{1}}^{p}\cdots x_{i_{r}}%
^{p}\right)  ^{1/p}. \label{in2}%
\end{equation}
Now, letting $\mathbf{y}:=\left(  x_{1}^{p},\ldots,x_{n}^{p}\right)  $,
Theorem \ref{th4.1} implies that
\[
\sum_{\left\{  i_{1},\ldots,i_{r}\right\}  \in E\left(  G\right)  }x_{i_{1}%
}^{p}\cdots x_{i_{r}}^{p}\leq\left(  \binom{n}{3}-k\binom{n/k}{3}\right)
\frac{1}{n^{3}}.
\]
On the other hand, $G$ must be complete $k$-chromatic, say with vertex classes
of sizes $n_{1},\ldots,n_{k}.$ Define the function
\[
f\left(  x\right)  =\left\{
\begin{array}
[c]{ll}%
0, & \text{if }x\leq2\\
\binom{x}{3}, & \text{if }x>2
\end{array}
\right.  ,
\]
and note that $f\left(  x\right)  $ is convex. Therefore
\begin{align*}
e\left(  G\right)   &  =\binom{n}{3}-\sum_{i=1}^{k}\binom{n_{i}}{3}=\binom
{n}{3}-\sum_{i=1}^{k}f\left(  n_{i}\right)  \leq\binom{n}{3}-kf\binom{n/k}%
{3}\\
&  =\binom{n}{3}-k\binom{n/k}{3}.
\end{align*}
We used above the fact that $n>2k.$

Hence, replacing $e\left(  G\right)  $ in the right side of (\ref{in2}), we
find that
\begin{align*}
\lambda^{\left(  p\right)  }\left(  G\right)   &  \leq3!\left(  \binom{n}%
{3}-k\binom{n/k}{3}\right)  ^{1-1/p}\left(  \binom{n}{3}-k\binom{n/k}%
{3}\right)  ^{1/p}n^{-3/p}\\
&  =3!\left(  \binom{n}{3}-k\binom{n/k}{3}\right)  n^{-3/p}.\text{ }%
\end{align*}

If equality holds, then we must have also
\[
\sum_{i=1}^{k}f\left(  n_{i}\right)  =kf\left(  n/k\right)  .
\]
An easy inspection shows this can happen only if $n_{1}=\cdots=n_{k}=n/k$.
Therefore, $k|n$ and $G=Q_{k}^{r}\left(  n\right)  .$
\end{proof}

\subsection{Proof of Theorem \ref{th5}}

\begin{proof}
The key to our proof is an appropriate bound on $\lambda^{\left(  1\right)
}\left(  G\right)  $. Let $G$ be $k$-chromatic $r$-graph of order $n$ with
maximum $\lambda^{\left(  1\right)  }$. Propositions \ref{pro3} and \ref{pro2}
imply that $G$ is complete $k$-chromatic; let $V_{1},\ldots,V_{k}$ be the
vertex sets of $G;$ for each $i\in\left[  k\right]  ,$ set $\left\vert
V_{i}\right\vert =n_{i}.$

Let $\left[  x_{i}\right]  $\textbf{ }be a nonnegative eigenvector to
$\lambda^{\left(  1\right)  }\left(  G\right)  ;$ Proposition \ref{eqt1}
implies that all entries belonging to the same partition set are equal, so for
each $i\in\left[  k\right]  ,$ let $a_{i}$ be the value of the entries in
$V_{i}.$

Write $V^{\left[  r\right]  }$ for the set of $r$-permutations of $V,$ that is
to say, the set of all vectors $\left(  i_{1},\ldots,i_{r}\right)  \in
V^{r\text{ }}$ with distinct entries. Note that%
\[
P_{G}\left(  \left[  x_{i}\right]  \right)  =\sum\left\{  x_{i_{1}}\cdots
x_{i_{r}}:\left(  i_{1},\ldots,i_{r}\right)  \in V^{\left[  r\right]  }\text{
and }\left(  i_{1},\ldots,i_{r}\right)  \notin V_{j}^{r},\text{ }%
j=1,\ldots,k\right\}  .
\]
Hence, we see that
\begin{align*}
P_{G}\left(  \left[  x_{i}\right]  \right)   &  \leq\sum\left\{  x_{i_{1}%
}\cdots x_{i_{r}}:\left(  i_{1},\ldots,i_{r}\right)  \in V^{r\text{ }}\text{
and }\left(  i_{1},\ldots,i_{r}\right)  \notin V_{j}^{r},\text{ }%
j=1,\ldots,k\right\} \\
&  =\left(  \sum_{i=1}^{n}x_{i}\right)  ^{r}-\sum_{j=1}^{k}\sum\left\{
x_{i_{1}}\cdots x_{i_{r}}:\left(  i_{1},\ldots,i_{r}\right)  \in V_{j}%
^{r}\right\} \\
&  =1-\sum_{j=1}^{k}\left(  n_{j}a_{j}\right)  ^{r}\leq1-k\left(  \frac{1}%
{k}\sum_{j=1}^{k}n_{j}a_{j}\right)  ^{r}=1-k^{-r+1}.
\end{align*}
Therefore,%
\[
\lambda^{\left(  1\right)  }\left(  G\right)  \leq1-k^{-r+1}.
\]
Let now $p>1,$ and let $\left[  x_{i}\right]  $\textbf{ }be a nonnegative
eigenvector to $\lambda^{\left(  p\right)  }\left(  G\right)  .$ The PM
inequality implies that
\[
\lambda^{\left(  p\right)  }\left(  G\right)  =r!\sum_{\left\{  i_{1}%
,\ldots,i_{r}\right\}  \in E\left(  G\right)  }x_{i_{1}}\cdots x_{i_{r}}\leq
r!e\left(  G\right)  ^{1-1/p}\left(  \sum_{\left\{  i_{1},\ldots
,i_{r}\right\}  \in E\left(  G\right)  }x_{i_{1}}^{p}\cdots x_{i_{r}}%
^{p}\right)  ^{1/p}.
\]
Now, letting $\mathbf{y}:=\left(  x_{1}^{p},\ldots,x_{n}^{p}\right)  ,$ we see
that
\[
\sum_{\left\{  i_{1},\ldots,i_{r}\right\}  \in E\left(  G\right)  }x_{i_{1}%
}^{p}\cdots x_{i_{r}}^{p}\leq\frac{\lambda^{\left(  1\right)  }\left(
G\right)  }{r!}\leq\frac{\left(  1-k^{-r+1}\right)  }{r!},
\]
implying inequality (\ref{in4}). Now, to get inequality (\ref{in5}), notice
that
\[
r!e\left(  G\right)  \leq r!\binom{n}{r}-kr!\binom{n/k}{r}<n^{r}\left(
1-\frac{1}{k^{-r+1}}\right)  ,
\]
and (\ref{in5}) follows from (\ref{in4}).
\end{proof}

\subsection{Acknowledgement}

Part of this work has been done, while the second author was visiting Shanghai
University and Hong Kong Polytechnic University in the Fall of 2013. He is
grateful for the outstanding hospitality of these institutions.\ The research
of the first author was supported by some grants (No. 11171207, No. 91130032),
and by a grant of \textquotedblleft The First-class Discipline of Universities
in Shanghai\textquotedblright. The research of the third author was supported
by National Science Foundation of China (No. 11101263), and by a grant of
\textquotedblleft The First-class Discipline of Universities in
Shanghai\textquotedblright.

\end{document}